\documentclass[12pt]{amsart}
\usepackage{amsmath,mathtools,amsthm,amsfonts,bigints,amssymb, mathrsfs, tikz-cd, xcolor}
\usepackage{tabularx}
\usepackage{multirow}
\newtheorem{theorem}{Theorem}[section]
\newtheorem{lemma}[theorem]{Lemma}

\theoremstyle{definition}

\newtheorem{remark}[theorem]{Remark}

\newcommand{\what}{\widehat}

\newcommand{\R}{\mathbb R}%
\newcommand{\C}{\mathbb C}%
\newcommand{\N}{\mathbb N}%
\newcommand{\X}{\mathbb X}%
\newcommand{\La}{\mathcal L}%
\newcommand{\Hc}{\mathcal H}%
\newcommand{\Qc}{\mathcal Q}%
\newcommand{\Pc}{\mathcal P}%

\newcommand{\F}{\mathscr F}%
\numberwithin{equation}{section}
\makeatletter

\renewcommand\subsubsection{\@secnumfont}{\bfseries}%
\renewcommand\subsubsection{\@startsection{subsubsection}{3}
  \z@{.5\linespacing\@plus.7\linespacing}{-.5em}%
  {\normalfont\bfseries}}

  \makeatother
\makeatletter
\@namedef{subjclassname@2020}{%
  \textup{2020} Mathematics Subject Classification}
\makeatother
\begin{document}

\title[Iterated convolution inequalities]{Iterated convolution inequalities on $\mathbb{R}^d$ and Riemannian Symmetric Spaces of non-compact type}

\author[Utsav Dewan]{Utsav Dewan}
\address{Stat-Math Unit, Indian Statistical Institute, 203 B. T. Rd., Kolkata 700108, India}
\email{utsav\_r@isical.ac.in\:,\: utsav97dewan@gmail.com}



\begin{abstract}
In a recent work (Int Math Res Not 24:18604-18612, 2021), Carlen-Jauslin-Lieb-Loss studied the  convolution inequality $f \ge f*f$ on $\mathbb{R}^d$ and proved that the real integrable solutions of the above inequality must be non-negative and satisfy the non-trivial bound $\int_{\R^d} f \le \frac{1}{2}$\:. Nakamura-Sawano then generalized their result to $m$-fold convolution (J Geom Anal 35:68, 2025). In this article, we replace the monomials by genuine polynomials and study the real-valued solutions $f \in L^1(\mathbb{R}^d)$ of the iterated convolution inequality  
\begin{equation*} 
f \ge \displaystyle\sum_{n=2}^N a_n \left(*^n f\right) \:,
\end{equation*}
where $N \ge 2$ is an integer and for $2 \le n \le N$, $a_n$ are non-negative integers with at least one of them positive. We prove that $f$ must be non-negative and satisfy the non-trivial bound $\int_{\mathbb{R}^d} f \le t_{\mathcal{Q}}\:$ where $\mathcal{Q}(t):=t-\displaystyle\sum_{n=2}^N a_n\:t^n$ and $t_{\mathcal{Q}}$ is the unique zero of $\mathcal{Q}'$ in $(0,\infty)$. We also have an analogue of our result for Riemannian Symmetric Spaces of non-compact type. Our arguments involve Fourier Analysis and Complex analysis. We then apply our result to obtain an a priori estimate for solutions of an integro-differential equation which is related to the physical problem of the ground state energy of the Bose gas in the classical Euclidean setting.

\end{abstract}

\maketitle
\tableofcontents

\section{Introduction}
While conducting a recent study on the physical problem of the ground state energy of the Bose gas in \cite{CJL}, Carlen-Jauslin-Lieb were propelled to look at the real, integrable solutions of the following convolution inequality almost everywhere on $\R^d$:
\begin{equation} \label{two-conv}
f \ge f * f\:.
\end{equation}
While it is easy to observe that the equality case of (\ref{two-conv}) fails to produce any non-trivial integrable solution, in \cite{CJLL} Carlen-Jauslin-Lieb-Loss produced a family of non-trivial radial solutions in $L^1(\R^d)$ for the inequality (\ref{two-conv}) by considering
\begin{equation} \label{poisson-ex}
f_{a,t} = \widehat{g_{a,t}}\:,
\end{equation}
the Fourier transform of the functions
\begin{equation*}
g_{a,t}(x):=ae^{-2\pi t\|x\|}\:, \text{ for } 0 < a \le 1/2\:,\: t>0\:.
\end{equation*} 
They observed that the functions $f_{a,t}$ are non-negative and satisfy
\begin{equation*} 
\int_{\R^d} f_{a,t} \le \frac{1}{2}\:.
\end{equation*}
Then the authors proceeded to show that the above  examples $f_{a,t}$ are surprisingly typical of {\em all} integrable solutions by proving \cite[Theorem 1]{CJLL}: any real $f \in L^1(\R^d)$ satisfying the convolution inequality (\ref{two-conv}) is non-negative with
\begin{equation} \label{two-conv-int-bound}
\int_{\R^d} f \le \frac{1}{2}\:.
\end{equation}
We note that by simply integrating out the inequality (\ref{two-conv}), one only gets the a priori bound
\begin{equation*} 
\int_{\R^d} f \le 1\:,
\end{equation*}
and thus (\ref{two-conv-int-bound}) is non-trivial. The non-negativity of $f$ is also special to the case $L^1(\R)$ as for any $p>1$, one can get counter-examples belonging to $L^p(\R)$ by simply considering the Fourier transform of the indicator function of symmetric intervals centred at the origin. 

Recently, Nakamura-Sawano generalized the above result of Carlen et. al. by studying the integrable solutions of the $m$-fold convolution inequality on $\R^d$ for $m \ge 2$:
\begin{equation} \label{m-conv}
f \ge \underbrace{f * \cdots *f }_{m-times}\:.
\end{equation}
Under the additional condition that $\int_{\R^d} f \ge 0$, if $m$ is odd, they showed that \cite[Theorems 1.2, 1.4]{NS}: any real $f \in L^1(\R^d)$ satisfying the convolution inequality (\ref{m-conv}) is non-negative with
\begin{equation} \label{m-conv-int-bound}
\int_{\R^d} f \le m^{-\frac{1}{m-1}}\:.
\end{equation} 

An inquisitive mind is naturally led to inquire what happens when the iterated convolution inequalities (\ref{two-conv}) or (\ref{m-conv}) are replaced by a genuine polynomial:
\begin{equation} \label{polynomial-conv}
f \ge \displaystyle\sum_{n=2}^N a_n \left(*^n f\right) \:,
\end{equation}
where $N \ge 2$ is an integer, for $2 \le n \le N$, $a_n$ are non-negative integers with at least one of them positive and 
\begin{equation*}
*^n f = \underbrace{f * \cdots *f }_{n-times}\:.
\end{equation*}

In this article, we study the convolution inequality (\ref{polynomial-conv}) on $\R^d$ and Riemannian Symmetric Spaces of non-compact type. Intimately connected to our analysis of the inequality (\ref{polynomial-conv}), will be the polynomial
\begin{equation} \label{Q-defn}
\Qc(t)=t-\displaystyle\sum_{n=2}^N a_n\:t^n\:,
\end{equation}
where $a_n$ are as in (\ref{polynomial-conv}). It is easy to observe that $\Qc$ attains a unique positive maximum on $(0,\infty)$ say at $t_{\Qc}$ and  
\begin{equation} \label{tQ_prop1}
t_{\Qc} \in (0,1)\:\:\text{ with }\:\Qc\left(t_{\Qc}\right)<t_{\Qc}\:.
\end{equation}
Moreover, by the Descartes' rule of signs, the polynomial $\Qc'$ has a unique zero in $(0,\infty)$ and thus is given by $t_{\Qc}$.

Upon a brief glance at the inequality (\ref{polynomial-conv}), one may inquire whether the class of real-valued $f \in L^1(\R^d)$ satisfying the iterated convolution inequality (\ref{polynomial-conv}) to begin with, is sufficiently rich. Our first result (Theorem \ref{thm}) shows that it is indeed the case for $G$ a second countable, locally compact, Hausdorff topological group and in fact, one can have a concrete realization of the structure of such solutions. In the following statement $L^1(G)$ is considered with respect to a fixed left Haar measure.
\begin{theorem} \label{thm}
Let $N \ge 2$ be an integer, for $2 \le n \le N$, $a_n$ be non-negative integers with at least one of them positive, $\Qc$ be as in (\ref{Q-defn}) and $t_{\Qc}$ be the unique zero of $\Qc'$ in $(0,\infty)$. Let $G$ be a second countable, locally compact, Hausdorff topological group. Now consider any non-negative $\psi \in L^1(G)$ with 
\begin{equation*}
\|\psi\|_{L^1(G)} \le \Qc\left(t_{\Qc}\right)\:.
\end{equation*}
Define $\Psi_j$, $j \in \N \cup \{0\}$, inductively by,
\begin{eqnarray*}
&&\Psi_0 := \psi\:, \\
&&\Psi_{j+1} := \psi \:+\:\displaystyle\sum_{n=2}^N a_n \left(*^n \Psi_j\right) \:,\:j \in \N \cup \{0\}\:.
\end{eqnarray*}
Then $\{\Psi_j\}_{j=0}^\infty$ converges to an element  $\Psi$ in the topology of $L^1(G)$ which satisfies 
\begin{equation*}
\Psi \ge 0\:,\: \|\Psi\|_{L^1(G)} \le t_{\Qc}\text{ and }\psi = \Psi - \displaystyle\sum_{n=2}^N a_n \left(*^n \Psi\right)\:.
\end{equation*}
\end{theorem}

Next, we study the properties of the solutions of the inequality (\ref{polynomial-conv}):
\begin{theorem} \label{Rn-thm}
Let $N,\:a_n,\:\Qc,\:t_{\Qc}$ be as in Theorem \ref{thm}. Suppose a real-valued $f \in L^1(\R^d)$ with $\int_{\R^d} f \ge 0\:,$ satisfies 
\begin{equation} \label{Rn-thm-eq}
f \ge \displaystyle\sum_{n=2}^N a_n \left(*^n f\right) \:.
\end{equation}
Then $f$ is non-negative and satisfies $\|f\|_{L^1(\R^d)}\le t_{\Qc}$\:.
\end{theorem}
Theorem \ref{thm} follows from a combinatorial argument. The proof of Theorem \ref{Rn-thm} involves Fourier Analysis. But the crux of generalizing the results (obtained in \cite{CJLL, NS}) from monomials to  genuine polynomials, lies in the application of a non-trivial result from complex analysis regarding singularities of power series, known as Vivanti-Pringsheim theorem (Lemma \ref{pringsheim}).

We also have an analogue of Theorem \ref{Rn-thm} for Riemannian Symmetric Spaces of non-compact type. These are homogeneous spaces $\X=G/K$, where $G$ is a connected, non-compact, semi-simple Lie group with finite center and $K$ is a maximal compact subgroup of $G$. All integrals or Lebesgue spaces considered on $G$ will be with respect to a fixed Haar measure, which we denote by $dx$.
\begin{theorem} \label{thm1}
Let $N,\:a_n,\:\Qc,\:t_{\Qc}$ be as in Theorem \ref{thm}. Suppose a right $K$-invariant, real-valued $f \in L^1(G)$ with $\int_G f \ge 0\:,$ satisfies 
\begin{equation} \label{thm1eq}
f \ge \displaystyle\sum_{n=2}^N a_n \left(*^n f\right) \:.
\end{equation}
Then $f$ is non-negative and satisfies $\|f\|_{L^1(G)}\le t_{\Qc}$\:.
\end{theorem}

Proof of Theorem \ref{thm1} necessitates an understanding of the operator valued Group Fourier transform on non-compact semi-simple Lie groups and utilizes the decay of the ground spherical function away from the identity. 

As an application of Theorem \ref{thm1}, we look at an integro-differential equation which is related to the physical problem of the ground state energy of the Bose gas in the classical Euclidean setting, considered in \cite{CJL}. To state our result, we need to introduce some notations. Let $\Delta$ be the Laplace-Beltrami operator corresponding to the left invariant Riemannian metric on $\X$. The $L^2$-spectrum of $\Delta$ is given by $(-\infty,-\|\rho\|^2]$ (for details about the notation, we refer to section $2$). We now consider the shifted Laplace-Beltrami operator $\La:= \Delta + \|\rho\|^2$ and the integro-differential equation of a $K$-biinvariant function $u$:

\begin{equation} \label{PDE}
{\left(-\La +\xi \right)}^m u(x) = V(x)(1-u(x))+\mu \left(*^{m+1} u\right)(x)\:,\: x \in G\:,
\end{equation}
under a constraint 
\begin{equation} \label{constraint}
\int_G u = \frac{1}{\delta}\:,
\end{equation}
where $m$ is a positive integer, $\mu \ge 0$\:,\:$\xi,\delta>0$ are given parameters and $V$ is a given $K$-biinvariant non-negative potential in $L^1(G)$. Then we have the a priori estimate:
\begin{theorem} \label{thm2}
Let $m \in \N,\: \mu \ge 0\:,\:\xi, \delta>0$ satisfy
\begin{equation} \label{thm2_eq}
\frac{\mu^{1/m}}{\xi\delta} <1\:.
\end{equation}
Suppose $V \in L^1(G//K)$ is non-negative and real $u \in L^1(G//K)$ is a solution of (\ref{PDE})-(\ref{constraint}) such that $u \le 1$, almost everywhere on $G$. Then $u$ is non-negative almost everywhere on $G$. 
\end{theorem}

This article is organized as follows. In section $2$, we recall the necessary preliminaries about Riemannian Symmetric Spaces of non-compact type and the Fourier analysis thereon. In section $3$, we prove Theorems \ref{thm} and \ref{Rn-thm}. In section $4$, we prove Theorems \ref{thm1} and \ref{thm2}.

\section{Preliminaries}
In this section, we recall the essential preliminaries and fix our notations. We refer the reader to \cite{H} for more details. 

Two non-negative functions $f_1$ and $f_2$ are defined to satisfy $f_1 \lesssim f_2$ if there exists a constant $C>0$ such that $f_1 \le Cf_2$\:. $\N$ will denote the set of positive integers. 

Let $G$ be a connected, non-compact, semi-simple Lie group with finite center and $K$ be a maximal compact subgroup of $G$. Then $\X=G/K$ is a Riemannian Symmetric Space of non-compact type. The class of $K$-biinvariant functions on $G$ are naturally identified with left $K$-invariant functions on $\X$.

Let $\mathfrak{g}$ be the Lie algebra of $G$ and $\mathfrak{g}=\mathfrak{k}+\mathfrak{p}$ be the Cartan decomposition on the Lie algebra level. The Killing form of $\mathfrak{g}$ induces a $K$-invariant inner product $\langle \cdot,\cdot \rangle$ on $\mathfrak{p}$ and hence a $G$-invariant Riemannian metric on $\X$.

We fix a maximal abelian subspace $\mathfrak{a}$ of $\mathfrak{p}$. The rank of $\X$ is the real rank of $G$ which is given by the dimension of $\mathfrak{a}$. $\mathfrak{a}^*$, the real dual of $\mathfrak{a}$ will be identified with $\mathfrak{a}$ via the inner product inherited from $\mathfrak{p}$. Let $\Sigma$ denote the the set of restricted roots of the pair $(\mathfrak{g}, \mathfrak{a})$ and $W$ be the Weyl group associated with $\Sigma$. For $\alpha \in \Sigma$, let $\mathfrak{g}_\alpha$ and $m_\alpha$ denote the root space corresponding to the root $\alpha$ and the dimension of $\mathfrak{g}_\alpha$ respectively. We make the choice of a positive Weyl chamber $\mathfrak{a}^+ \subset \mathfrak{a}$, and let $\Sigma^+$ and $\Sigma^+_r$ be the corresponding set of positive roots and positive reduced (indivisible) roots respectively. Let $\rho=\frac{1}{2}\displaystyle\sum_{\alpha \in \Sigma^+} m_\alpha \alpha$.

The Cartan decomposition of $G$ is given by
\begin{equation*}
G=K(\exp \overline{\mathfrak{a}^+})K\:.
\end{equation*}
For $x \in G$ we denote by $x^+ \in \overline{\mathfrak{a}^+}$ the corresponding component in the Cartan decomposition. We also have the Iwasawa decomposition given by $G=KAN$ and corresponding to the decomposition, we write for $g \in G$, $g=K(g)\exp H(g) N(g)$, where $K(g) \in K,\: N(g) \in N$ and $H(g) \in \mathfrak{a}$. Let $M$ be the centralizer of $A$ in $K$. 

For $\lambda \in \mathfrak{a}_\C$, one has the spherical principal series representation $\pi_\lambda$ of $G$ on $L^2(K/M)$ given by,
\begin{equation*}
\left(\pi_\lambda(x) V\right)(b):= e^{(i\lambda -\rho)H(x^{-1}b)}\:V(K(x^{-1}b))\:,\: \text{ for all } V \in L^2(K/M)\:,\: b \in K
\end{equation*}
(see \cite{GV} for more details).  We consider an orthonormal basis $\{e_j\}_{j=0}^\infty$ of $L^2(K/M)$ with the unique $K$-fixed vector $e_0 \equiv 1$. For $\lambda \in \mathfrak{a}$, the Group Fourier transform of $f \in L^1(X)$ defined by,
\begin{equation*}
\what{f}(\pi_\lambda):=\int_G f(x)\:\pi_\lambda(x)\:dx \:,
\end{equation*}
is a bounded linear operator on $L^2(K/M)$ and can be identified with the Helgason Fourier transform of $f$ via,
\begin{equation*}
\what{f}(\pi_\lambda)(1)(b)=\int_G e^{(i\lambda - \rho)H(x^{-1}b)}\: f(x)\:dx = \tilde{f}(\lambda,b)\:,\:\text{ for } b \in K/M\:.
\end{equation*}
The Helgason Fourier transform of $f$ is injective on $L^1(G/K)$, that is, if $\what{f}(\pi_\lambda)=0$ for all $\lambda \in \mathfrak{a}$ then $f=0$ almost everywhere \cite[Theorem 1.9, Chapter 3, p. 213]{H}.

For $\lambda \in \mathfrak{a}$, the Spherical function $\varphi_\lambda$ is given by,
\begin{equation*}
\varphi_\lambda(g)= \int_K e^{(-i\lambda + \rho)A(kg)}\: dk\:,
\end{equation*}
where $A(g)=-H(g^{-1})$. This function is identified as the following matrix coefficient of the principal series representation $\pi_\lambda$,
\begin{equation*}
\varphi_\lambda(g) = \left\langle \pi_\lambda(g)e_0\:,\:e_0\right\rangle\:.
\end{equation*}
On $\mathfrak{a}$, $\|\cdot\|$ denotes the norm inherited from $\mathfrak{p}$. For $\lambda \in \mathfrak{a}$, the spherical function $\varphi_\lambda$ is a smooth $K$-biinvariant eigenfunction of the Laplace-Beltrami operator $\Delta$,
\begin{equation*}
\Delta \varphi_\lambda = -\left(\|\lambda\|^2 + \|\rho\|^2\right) \varphi_\lambda\:.
\end{equation*}
For $\lambda \in \mathfrak{a}$, one has
\begin{equation} \label{phi_lambda_bound}
|\varphi_\lambda(x)| \le 1\:,\:\text{ for all } x \in G\:.
\end{equation}
In fact, one has the following bounds which are decay estimates away from the identity \cite[Proposition 2.2.12]{AJ}:
\begin{equation} \label{phi_lambda_decay}
|\varphi_\lambda(x)| \le  \varphi_0(x) \lesssim \left\{\displaystyle\prod_{\alpha \in \Sigma^+_r} \left(1+\langle\alpha,x^+\rangle\right)\right\}e^{-\langle\rho,x^+\rangle}
\:,\:\text{ for all } x \in G\:.
\end{equation}

For $\lambda \in \mathfrak{a}$ and $f \in L^1(G//K)$, the Spherical Fourier transform of $f$,
\begin{equation*}
 \Hc f(\lambda)= \int_G f(x) \:\varphi_\lambda(x)\:dx\:,
\end{equation*}
is naturally identified with the matrix coefficient of the Group Fourier transform given by $\left\langle \what{f}\left(\pi_\lambda\right)e_0\:,\:e_0\right\rangle$\:. 
$\Hc f$ is continuous on $\mathfrak{a}$ and also satisfies the Riemann-Lebesgue Lemma \cite[Theorem 1.8, Chapter 3, p. 209]{H}:
\begin{equation} \label{Riemann-Lebesgue}
\displaystyle\lim_{\|\lambda\| \to \infty} \Hc f(\lambda)=0\:.
\end{equation}

The $L^2$-spectrum of $\Delta$ is given by $(-\infty,-\|\rho\|^2]$. We consider the shifted Laplace-Beltrami operator $\La:= \Delta + \|\rho\|^2$, with $L^2$-spectrum $(-\infty,0]$. Then for $\xi >0$, the resolvent operator $\left(\xi I - \La\right)^{-1}$ is realized by convolution on the right with the $K$-biinvariant tempered distribution $K_\xi$ on $G$ given by \cite[pp. 279-280]{A},
\begin{equation*} 
K_\xi(x)= \int_0^\infty e^{-t\xi}\:h_t(x)\:dt\:,
\end{equation*}
where $h_t$ is the heat kernel on $\X$.

We will also require the following result from complex analysis which is known as Vivanti-Pringsheim theorem:
\begin{lemma} \cite[p. 235]{R} \label{pringsheim}
Let the power series $g(z)=\displaystyle\sum_{\nu=0}^\infty b_{\nu}\:z^{\nu}$ have positive finite radius of convergence $R$ and suppose that all but finitely many of its coefficients $b_{\nu}$ are real and non-negative. Then $z_0:=R$ is a singular point of $g$.
\end{lemma}

\section{Proofs of Theorem \ref{thm} and the Euclidean result}
In this section, we prove Theorems \ref{thm}  and \ref{Rn-thm}.
\begin{proof}[Proof of Theorem \ref{thm}]
Let 
\begin{equation*}
N_1 = \min\{n \in [2,N] \mid a_n>0\}\:\:\text{ and }\:\:N_2 = \max\{n \in [2,N] \mid a_n>0\}\:.
\end{equation*}
We will inductively show that for each $j \in \N \cup \{0\}$, there exists a sequence $\{m_{j,l}\}_{l=1}^\infty$ of non-negative integers such that
\begin{equation} \label{finite-series}
\Psi_j = \displaystyle\sum_{l=1}^\infty m_{j,l} \left(*^l \psi\right)\:,
\end{equation}
satisfying
\begin{equation} \label{item-i}
m_{j,l}=0\:,\:\text{ for } l \ge N^j_2+1\:.
\end{equation}
To prove $(\ref{item-i})$, we see that it is true for $j=0$ as 
\begin{equation*}
m_{0,l} =
\begin{cases}
	 1\:, \text{ if } l=1 \\
	0\:, \text{ if } l\ge 2  \:.
	\end{cases}
\end{equation*}
Let us assume $(\ref{item-i})$ to be true for $j=k$ and then prove it for $j=k+1$. By the iterative definition,
\begin{equation*}
\Psi_{k+1} = \psi \:+\:\displaystyle\sum_{n=2}^N a_n \left(*^n \Psi_k\right)\:.
\end{equation*}
Now using the induction hypothesis, we can write
\begin{equation*}
\Psi_{k+1} = \psi \:+\:\displaystyle\sum_{n=N_1}^{N_2} a_n \left[*^n \left\{\displaystyle\sum_{l=1}^{N^k_2}m_{k,l}\:\left(*^l \psi\right)\right\}\right]= \displaystyle\sum_{l=1}^{N^{k+1}_2} b_l \left(*^l \psi\right)\:,
\end{equation*}
for some non-negative integers $b_l$ for $l \in [1,N^{k+1}_2]$\:. This gives $(\ref{item-i})$.

We now claim that 
\begin{equation} \label{item-ii}
m_{j+1,l} \ge m_{j,l}\:\text{ for all }l \in \N\:.
\end{equation}
To prove $(\ref{item-ii})$, we see that it is true for $l=1$ as $N_1 \ge 2$ and hence $m_{j,1}=1$ for all $j \in \N \cup \{0\}$. Now assuming $(\ref{item-ii})$ to be true for $l \le k$, we get $(\ref{item-ii})$ for $l=k+1$ by noting that
\begin{equation*}
\Psi_{j+2} - \Psi_{j+1} = \displaystyle\sum_{n=N_1}^{N_2} a_n \displaystyle\sum_{p=0}^{n-1} \left(*^p \Psi_j\right) * \left(\Psi_{j+1}-\Psi_j\right) * \left(*^{n-p-1} \Psi_{j+1}\right)
\end{equation*}
and the induction hypothesis.

By $(\ref{item-i})$ and $(\ref{item-ii})$, we get that $\{\Psi_j\}_{j=0}^\infty$ is an increasing sequence of non-negative integrable functions on $G$. We next claim that
\begin{equation}\label{item-iii}
m_{j,l} = m_{j+1,l}\:\text{ for all }j \ge l\:.
\end{equation}
To prove $(\ref{item-iii})$, we note that
\begin{equation*}
\Psi_1 - \Psi_0 = \displaystyle\sum_{n=N_1}^{N_2} a_n \left(*^n \psi\right) \:,
\end{equation*}
and for $j \ge 1$,
\begin{eqnarray*}
\Psi_{j+1} - \Psi_j &=& \displaystyle\sum_{n=N_1}^{N_2} a_n \displaystyle\sum_{p=0}^{n-1} \left(*^p \Psi_{j-1}\right) * \left(\Psi_j-\Psi_{j-1}\right) * \left(*^{n-p-1} \Psi_j\right) \\
&=& \left(\Psi_j-\Psi_{j-1}\right)*\displaystyle\sum_{n=N_1}^{N_2} a_n \displaystyle\sum_{p=0}^{n-1} \left(*^p \Psi_{j-1}\right) * \left(*^{n-p-1} \Psi_j\right) \\
&=& \left(\Psi_j-\Psi_{j-1}\right)*\displaystyle\sum_{n=N_1}^{N_2} a_n \left\{n \left(*^{n-1} \psi \right) + \cdots\right\} \\
&=& \left(\Psi_j-\Psi_{j-1}\right)* \left\{N_1a_{N_1}\left(*^{N_1-1} \psi \right)+\cdots\right\} \\
&=& \left(\Psi_1-\Psi_0\right)* \left\{\left(N_1a_{N_1}\right)^{j-1}\left(*^{(N_1-1)(j-1)} \psi \right)+\cdots\right\} \\
&=& \left\{a_{N_1}\left(*^{N_1} \psi \right)+\cdots\right\}* \left\{\left(N_1a_{N_1}\right)^{j-1}\left(*^{(N_1-1)(j-1)} \psi \right)+\cdots\right\} \\
&=& N_1^{j-1}a^j_{N_1}\left(*^{(N_1j-j+1)} \psi \right)+\cdots\:.
\end{eqnarray*}
Now as $N_1 \ge 2$, $(\ref{item-iii})$ follows.

From $(\ref{item-iii})$, it follows that $\Psi_j \uparrow \Psi$ where 
\begin{equation} \label{soln}
\Psi = \displaystyle\sum_{l=1}^\infty m_{l,l} \left(*^l \psi\right)\:.
\end{equation}
Finally, we will show that 
\begin{equation} \label{item-iv}
\|\Psi_j\|_{L^1(G)} = \displaystyle\sum_{l=1}^{N^j_2}  m_{j,l}\:\|\psi\|^l_{L^1(G)}\le \displaystyle\sum_{l=1}^{N^j_2}  m_{j,l}\:\Qc(t_{\Qc})^l \le t_{\Qc}\:.
\end{equation}
Assuming $(\ref{item-iv})$, we see that the series in (\ref{soln}) converges absolutely. Then the desired properties of $\Psi$ follow. So we are left to prove $(\ref{item-iv})$. To show $(\ref{item-iv})$ we note that for $j=0$, the assertion is true as by (\ref{tQ_prop1}), it follows that
\begin{equation*}
\|\Psi_0\|_{L^1(G)} = \|\psi\|_{L^1(G)} \le \Qc(t_{\Qc}) < t_{\Qc}\:.
\end{equation*}
For general $j \ge 1$, by construction and the hypothesis, the equality and the first inequality in (\ref{item-iv}) follows. So we assume
\begin{equation} \label{final-estimate}
\displaystyle\sum_{l=1}^{N^j_2}  m_{j,l}\:\Qc(t_{\Qc})^l \le t_{\Qc}\:,
\end{equation}
to be true for $j=k$ and show it to be true for $j=k+1$. Now consider a non-negative $\psi \in L^1(G)$ with $\int_G\psi = \Qc\left(t_{\Qc}\right)$. Then defining $\{\Psi_j\}_{j=0}^\infty$ iteratively and then integrating out the relation,
\begin{equation*}
\Psi_{k+1} = \psi \:+\:\displaystyle\sum_{n=2}^N a_n \left(*^n \Psi_k\right)\:,
\end{equation*}
we get that
\begin{equation} \label{final-estimate-eq1}
\int_G \Psi_{k+1} = \int_G \psi  +\displaystyle\sum_{n=2}^N a_n \left(\int_G \Psi_k\right)^n\:.
\end{equation}
Now using (\ref{finite-series}) and (\ref{item-i}), we get
\begin{equation} \label{final-estimate-eq2}
\int_G \Psi_{k+1} = \displaystyle\sum_{l=1}^{N^{k+1}_2} m_{k+1,l}\: \left(\int_G \psi\right)^l = \displaystyle\sum_{l=1}^{N^{k+1}_2} m_{k+1,l}\: \Qc(t_{\Qc})^l\:,
\end{equation}
and also
\begin{eqnarray} \label{final-estimate-eq3}
\displaystyle\sum_{n=2}^N a_n \left(\int_G \Psi_k\right)^n &=& \displaystyle\sum_{n=2}^N a_n \left[\bigintssss_G \left\{\displaystyle\sum_{l=1}^{N^k_2} m_{k,l} \left(*^l \psi\right)\right\}\right]^n \nonumber\\
&=& \displaystyle\sum_{n=2}^N a_n \left[\displaystyle\sum_{l=1}^{N^k_2} m_{k,l} \left(\bigintssss_G \psi\right)^l\right]^n \nonumber\\
&=& \displaystyle\sum_{n=2}^N a_n \left(\displaystyle\sum_{l=1}^{N^k_2} m_{k,l} \Qc(t_{\Qc})^l\right)^n\:.
\end{eqnarray}
Then plugging (\ref{final-estimate-eq2}) and (\ref{final-estimate-eq3}) in (\ref{final-estimate-eq1}) and using (\ref{final-estimate}) for $j=k$ along with the fact that the polynomial $\Pc(t):=\displaystyle\sum_{n=2}^N a_n\:t^n$ is strictly increasing on $(0,\infty)$, it follows that 
\begin{equation*}
\displaystyle\sum_{l=1}^{N^{k+1}_2} m_{k+1,l}\: \Qc(t_{\Qc})^l = \Qc(t_{\Qc}) + \displaystyle\sum_{n=2}^N a_n \left(\displaystyle\sum_{l=1}^{N^k_2} m_{k,l} \Qc(t_{\Qc})^l\right)^n \le \Qc(t_{\Qc}) + \displaystyle\sum_{n=2}^N a_n\:t^n_{\Qc} = t_{\Qc}\:.
\end{equation*}
This completes the proof of (\ref{item-iv}) and hence also of Theorem \ref{thm}.
\end{proof}
We now present the proof of Theorem \ref{Rn-thm}.
\begin{proof}[Proof of Theorem \ref{Rn-thm}]
Let $f$ be as in the hypothesis of Theorem \ref{Rn-thm}. We now define,
\begin{equation} \label{Rn-pf-eq1}
\psi := f - \displaystyle\sum_{n=2}^N a_n \left(*^n f\right)\:.
\end{equation}
Then $\psi \in L^1(\R^d)$ is non-negative. Moreover, keeping in mind the hypothesis that $\int_{\R^d} f \ge 0$, we integrate the relation (\ref{Rn-pf-eq1}) and then use the fact that $\Qc$ attains a unique positive maximum on $(0,\infty)$ at $t_{\Qc}$ to obtain
\begin{equation*}
0 \le \int_{\R^d} \psi = \int_{\R^d} f - \displaystyle\sum_{n=2}^N a_n \left(\int_{\R^d} f\right)^n = \Qc\left(\int_{\R^d} f\right) \le \Qc(t_{\Qc})\:. 
\end{equation*}
Hence we also get
\begin{equation*}
\|\psi\|_{L^1(\R^d)} \le \Qc(t_{\Qc})\:. 
\end{equation*}
Then applying Theorem \ref{thm}, we get a $\Psi \in L^1(\R^d)$ such that 
\begin{equation} \label{Rn-pf-eq2}
\Psi \ge 0\:,\: \|\Psi\|_{L^1(\R^d)} \le t_{\Qc}\: \text{ and } \psi = \Psi - \displaystyle\sum_{n=2}^N a_n \left(*^n \Psi\right).
\end{equation}
Now if we can show that the Fourier transforms $\F f$ and $\F \Psi$ agree, that is,
\begin{equation} \label{Rn-pf-eq3}
\F f(\xi)= \F \Psi(\xi)\:,\: \text { for all } \xi \in \R^d\:,
\end{equation}
then by injectivity of the Fourier transform, it will follow that $f=\Psi$ as elements of $L^1(\R^d)$ and $f$ will inherit the desired properties from $\Psi$ stated in (\ref{Rn-pf-eq2}), which will then complete the proof of Theorem \ref{Rn-thm}. Thus it is enough to prove (\ref{Rn-pf-eq3}). 

As $\Psi$ is non-negative, we claim that 
\begin{equation} \label{Rn-pf-eq4}
\left|\F \Psi(\xi)\right| < t_{\Qc}\:,\:\text{ for all } \xi \in \R^d \setminus \{0\}\:.
\end{equation}
If $\Psi$ is a trivial element in $L^1(\R^d)$, that is, $\Psi$ is zero almost everywhere, then trivially
\begin{equation*}
\left|\F \Psi(\xi)\right|=0< t_{\Qc}\:,\:\text{ for all } \xi \in \R^d \:.
\end{equation*}
To see (\ref{Rn-pf-eq4}) for $\Psi$ non-trivial, by non-negativity of $\Psi$, we have for all $\xi \in \R^d$,
\begin{equation} \label{Rn-pf-eq5}
\left|\F \Psi(\xi)\right| = \left|\int_{\R^d} \Psi(x)\:e^{-2\pi ix\cdot\xi}\:dx\right| \le  \int_{\R^d} \Psi(x)\:dx\:.
\end{equation}
We now find out for which $\xi$, one can actually expect to get equality in (\ref{Rn-pf-eq5}). Equality above and the fact that $\Psi>0$ on a set of positive Lebesgue measure would yield a constant $\alpha \in \C$ such that
\begin{equation*}
\alpha\:e^{-2\pi ix\cdot\xi}=1\:,
\end{equation*}
holds on a set of positive Lebesgue measure and hence by real-analyticity, actually holds on $\R^d$. Then for $1\le j \le d$, considering $j$th partial derivatives $\frac{\partial}{\partial x_j}$, we get
\begin{equation*}
-2\pi i \xi_j\: e^{-2\pi ix\cdot\xi}=0\:,
\end{equation*} 
which implies $\xi_j=0$ for all $1\le j \le d$. Thus (\ref{Rn-pf-eq5}) is strict inequality unless $\xi=0$. Then by (\ref{Rn-pf-eq2}),
\begin{equation*}
\left|\F \Psi(\xi)\right| < \int_{\R^d} \Psi \le t_{\Qc}\:,\:\text{ for all } \xi \in \R^d \setminus \{0\}\:.
\end{equation*}
This proves the claim (\ref{Rn-pf-eq4}).

We now define
\begin{equation*}
\Omega:= \left\{\xi \in \R^d \setminus \{0\} \mid \F f(\xi)= \F \Psi(\xi)\right\}\subset \R^d \setminus \{0\}\:,
\end{equation*}
and by adopting a topological argument, we show that $\Omega = \R^d \setminus \{0\}$. Since both $\F f$ and $\F \Psi$ are continuous, it follows that $\Omega$ is closed in $\R^d \setminus \{0\}$. To show that  $\Omega$ is also open in $\R^d \setminus \{0\}$, we let $\xi_0 \in \Omega$. Then by (\ref{Rn-pf-eq4}), there exists $\varepsilon>0$ such that
\begin{equation} \label{Rn-pf-eq6}
\left|\F \Psi(\xi_0)\right| < (1-\varepsilon)t_{\Qc}\:.
\end{equation}
By continuity of $\F f$, there exists $\delta>0$ such that
\begin{equation}\label{Rn-pf-eq7}
\left|\left|\F f(\xi_0)\right|-\left|\F f(\xi)\right|\right| < \varepsilon\: t_{\Qc}\:,\:\text{ for all } \xi \in B(\xi_0,\delta)\:,
\end{equation}
where $B(\xi_0,\delta) \subset \R^d \setminus \{0\}$, is the Euclidean ball with center $\xi_0$ and radius $\delta$. Now as $\xi_0 \in \Omega$, by (\ref{Rn-pf-eq6}) we have
\begin{equation*}
\left|\F f(\xi_0)\right|=\left|\F \Psi(\xi_0)\right| < (1-\varepsilon)t_{\Qc}\:,
\end{equation*}
which combined with (\ref{Rn-pf-eq7}) yields
\begin{equation*}
\left|\F f(\xi)\right| < t_{\Qc}\:,\:\text{ for all } \xi \in B(\xi_0,\delta)\:.
\end{equation*}
So in particular both
\begin{equation}\label{Rn-pf-eq8}
\left|\F \Psi(\xi)\right|<t_{\Qc} \:\text{ and }\:\left|\F f(\xi)\right| <t_{\Qc}\:,\:\text{ for all } \xi \in B(\xi_0,\delta)\:.
\end{equation}
Now as by (\ref{Rn-pf-eq1}) and (\ref{Rn-pf-eq2}),
\begin{equation*}
f - \displaystyle\sum_{n=2}^N a_n \left(*^n f\right) = \psi = \Psi - \displaystyle\sum_{n=2}^N a_n \left(*^n \Psi\right)\:,
\end{equation*} 
by taking Fourier transform, we have
\begin{equation} \label{Rn-pf-eq9}
\F f - \displaystyle\sum_{n=2}^N a_n \left(\F f\right)^n =  \F \Psi - \displaystyle\sum_{n=2}^N a_n \left(\F \Psi\right)^n\:.
\end{equation}
Then in view of (\ref{Rn-pf-eq8}), we consider the complex polynomial $\Qc : \C \to \C$ and the disk 
\begin{equation*}
\mathcal{D} := \left\{z \in \C \mid |z|< t_{\Qc}\right\}\:.
\end{equation*}
We claim that $\Qc$ is injective on $\mathcal{D}$. To prove the claim, we write
\begin{equation*}
\Qc(z) = z - \displaystyle\sum_{n=2}^N a_n z^n = z - \Pc(z)\:.
\end{equation*}
Then
\begin{equation*}
\Qc'(z) = 1 - \displaystyle\sum_{n=2}^N na_n z^{n-1} = 1 - \Pc'(z)\:,
\end{equation*}
and hence $\Qc'(0)=1$ and $\Pc'(0)=0$. Now as $\Qc'$ is a polynomial in one complex variable, by the Fundamental theorem of Algebra, it has finitely many zeroes. Let $z_0$ be a zero of $\Qc'$ of smallest modulus. We now consider the disk 
\begin{equation*}
\mathcal{D}' := \left\{ z \in \C \mid |z| < |z_0| \right\}\:. 
\end{equation*}
Now as $z_0$ is a zero of $\Qc'$ of smallest modulus, $\Pc'(z) \ne 1$ for all $z \in \mathcal{D}'$. We focus on the line segment $(0,|z_0|) \subset \mathcal{D}' \cap (0,\infty)$. As the coefficients $a_n$ are non-negative with at least one strictly positive, it follows that  $\Pc'$ is strictly increasing on $(0,|z_0|)$. Combining it with the (already noted) facts that  $\Pc'(0)=0$ and $\Pc'(z) \ne 1$ for all $z \in \mathcal{D}'$, it follows that 
\begin{equation} \label{Rn-inj_eq1}
\Pc'(z)< 1 \:,\text{ for } z \in (0,|z_0|)\:.
\end{equation}
Then by non-negativity of the coefficients $a_n$ and (\ref{Rn-inj_eq1}), it follows that for all $z \in \mathcal{D}'$,
\begin{equation*}
\left|\Pc'(z)\right| \le \displaystyle\sum_{n=2}^N na_n |z|^{n-1} < 1 \:.
\end{equation*}
Then by the geometric series expansion, we can write $1/\Qc'$ in the form of a convergent power series in $\mathcal{D}'$,
\begin{equation*}
\frac{1}{\Qc'(z)}= \frac{1}{1 - \Pc'(z)}= 1 + \sum_{n=1}^\infty \left(\Pc'(z)\right)^n\:,
\end{equation*}
with real and non-negative coefficients and radius of convergence $|z_0|$. Then by Vivanti-Pringsheim theorem (Lemma \ref{pringsheim}), $|z_0|$ is a singularity of $1/\Qc'$. Hence, $|z_0| \in (0,\infty)$ is a zero of $\Qc'$. But since $t_{\Qc}$ is the unique zero of  $\Qc'$ in $(0,\infty)$, it follows that $|z_0|=t_{\Qc}$ and $\mathcal{D}=\mathcal{D}'$. Consequently, we get that 
\begin{equation}  \label{Rn-inj_eq2}
\left|\Pc'(z)\right| < 1\:,\text{ for all } z \in \mathcal{D}.
\end{equation}
Now for two distinct $z_1,z_2 \in \mathcal{D}$, we have
\begin{equation*}
\Qc(z_1)-\Qc(z_2) = (z_1-z_2) -\int_{z_2}^{z_1} \Pc'(\xi)\:d\xi\:, 
\end{equation*}
where the above line integral is along the line segment joining $z_2$ to $z_1$. Then by triangle inequality and (\ref{Rn-inj_eq2}), it follows that
\begin{equation*}
\left|\Qc(z_1)-\Qc(z_2) \right| \ge \left|z_1-z_2\right| -\left|\int_{z_2}^{z_1} \Pc'(\xi)\:d\xi \right|> \left|z_1-z_2\right| - \left|z_1-z_2\right| =0\:.
\end{equation*} 
This proves the claim that $\Qc$ is injective on $\mathcal{D}$.

 Now (\ref{Rn-pf-eq8}) can be rewritten as
\begin{equation*}
\F \Psi(\xi)\:,\:\F f(\xi) \in \mathcal{D}\:,\:\text{ for all } \xi \in B(\xi_0,\delta)\:,
\end{equation*}
which along with (\ref{Rn-pf-eq9}) and the injectivity of $\Qc$ on $\mathcal{D}$ yield,
\begin{equation*}
\F \Psi(\xi)=\F f(\xi) \:,\:\text{ for all } \xi \in B(\xi_0,\delta)\:,
\end{equation*}
that is, $B(\xi_0,\delta) \subset \Omega$ and hence $\Omega$ is also open in $\R^d \setminus \{0\}$. 

We finally show that $\Omega$ is non-empty. By the Riemann-Lebesgue lemma, there exists $\xi_0 \in \R^d \setminus \{0\}$ such that
\begin{equation*} 
\left|\F f(\xi_0)\right| <  t_{\Qc}\:.
\end{equation*} 
Thus by (\ref{Rn-pf-eq4}), we get that both $\F f(\xi_0), \F \Psi(\xi_0) \in \mathcal{D}$ and then by (\ref{Rn-pf-eq9}) and the injectivity of $\Qc$ on $\mathcal{D}$, we obtain  $\F f(\xi_0)= \F \Psi(\xi_0)$, that is, $\xi_0 \in \Omega$ and hence $\Omega$ is non-empty. Moreover, for the special when $d=1$, implementing the above argument involving Riemann-Lebesgue Lemma on both the connected components of $\R \setminus \{0\}$, we can also see that $\Omega$ has non-empty intersections with both the connected components of $\R \setminus \{0\}$. 

For $d \ge 2$, since $\Omega$ is a non-empty subset of $\R^d \setminus \{0\}$ which is both closed as well as open, we have $\Omega=\R^d \setminus \{0\}$ and thus (\ref{Rn-pf-eq3}) is true for $\R^d \setminus \{0\}$. For $d=1$, $\Omega$ is both closed as well as open in $\R \setminus \{0\}$ and moreover, has non-empty intersections with both the connected components of $\R \setminus \{0\}$ and thus (\ref{Rn-pf-eq3}) is true for $\R \setminus \{0\}$. The equality at $0$ then follows by continuity. This completes the proof of Theorem \ref{Rn-thm}.

\end{proof}

\begin{remark} \label{Rn-remark}
The non-negativity of the integral condition in Theorem \ref{Rn-thm} is necessary. For the counter-example, we refer to \cite[point (2) of Remark $1.3$]{NS}.
\end{remark}

\section{Results on Symmetric Spaces}
In this section, we prove Theorems \ref{thm1}  and \ref{thm2}.
\begin{proof}[Proof of Theorem \ref{thm1}]
Let $f$ be as in the hypothesis of Theorem \ref{thm1}. We now define,
\begin{equation} \label{pf1eq1}
\psi := f - \displaystyle\sum_{n=2}^N a_n \left(*^n f\right)\:.
\end{equation}
Then $\psi \in L^1(G/K)$ is non-negative. Moreover, keeping in mind the hypothesis that $\int_G f \ge 0$, integrating the relation (\ref{pf1eq1}) and then using properties of $t_{\Qc}$ we again get
\begin{equation*}
\|\psi\|_{L^1(G)} \le \Qc(t_{\Qc})\:. 
\end{equation*}
Then applying Theorem \ref{thm}, we get a $\Psi \in L^1(G/K)$ such that 
\begin{equation} \label{pf1eq2}
\Psi \ge 0\:,\: \|\Psi\|_{L^1(G)} \le t_{\Qc}\: \text{ and } \psi = \Psi - \displaystyle\sum_{n=2}^N a_n \left(*^n \Psi\right).
\end{equation}
Now by injectivity of the Helgason Fourier transform, it suffices to show that
\begin{equation} \label{pf1eq3}
\what f(\pi_\lambda)= \what \Psi(\pi_\lambda)\:,\: \text { for all } \lambda \in \mathfrak{a}\:,
\end{equation}
as then it will follow that $f=\Psi$ as elements of $L^1(G/K)$ and $f$ will inherit the desired properties from $\Psi$ stated in (\ref{pf1eq2}), which will then complete the proof of Theorem \ref{thm1}. Thus it is enough to prove (\ref{pf1eq3}). 

By (\ref{pf1eq1}) and (\ref{pf1eq2}), we have
\begin{equation*}
f - \displaystyle\sum_{n=2}^N a_n \left(*^n f\right)= \psi = \Psi - \displaystyle\sum_{n=2}^N a_n \left(*^n \Psi\right)\:,
\end{equation*}
which upon taking the Group Fourier transform becomes
\begin{equation} \label{pf1eq4}
\what f(\pi_\lambda) - \displaystyle\sum_{n=2}^N a_n \what f(\pi_\lambda)^n= \what \Psi(\pi_\lambda) - \displaystyle\sum_{n=2}^N a_n \what \Psi(\pi_\lambda)^n\:.
\end{equation}
The equality above is of bounded operators on the Hilbert space $L^2(K/M)$. We consider the orthonormal basis $\{e_j\}_{j=0}^\infty$ of $L^2(K/M)$ with the unique $K$-fixed vector $e_0 \equiv 1$. Now for $(p,q) \in \left(\N \cup \{0\}\right)^2$, the equality (\ref{pf1eq4}) yields the corresponding equality in matrix coefficients
\begin{equation} \label{pf1eq5}
\what f(\pi_\lambda)_{p,q} - \displaystyle\sum_{n=2}^N a_n \left(\what f(\pi_\lambda)^n\right)_{p,q}= \what \Psi(\pi_\lambda)_{p,q} - \displaystyle\sum_{n=2}^N a_n \left(\what \Psi(\pi_\lambda)^n\right)_{p,q}\:,
\end{equation}
where 
\begin{equation*}
\what f(\pi_\lambda)_{p,q} = \left\langle \what f(\pi_\lambda)e_q\:,\:e_p \right\rangle\:,
\end{equation*}
and so on. We first note that as $f$ is right $K$-invariant,
\begin{equation} \label{pf1eq6}
\what f(\pi_\lambda)_{p,q} = 0\:,\: \text{ if } (p,q) \in \left(\N \cup \{0\}\right) \times \N\:,
\end{equation}
and similarly for $\Psi$. We now claim that for all $n \ge 2$,
\begin{equation} \label{pf1eq7}
\left(\what f(\pi_\lambda)^n\right)_{p,q} = \what f(\pi_\lambda)_{p,q}\:\what f(\pi_\lambda)^{n-1}_{0,0}\:.
\end{equation}
We prove the above claim by induction. For $n=2$,
\begin{eqnarray*}
\left(\what f(\pi_\lambda)^2\right)_{p,q} = \left\langle \what f(\pi_\lambda)\left(\what f(\pi_\lambda)e_q\right)\:,\:e_p \right\rangle &=& \left\langle \what f(\pi_\lambda)\left(\displaystyle\sum_{r=0}^\infty \left\langle\what f(\pi_\lambda)e_q\:,\:e_r\right\rangle e_r \right)\:,\:e_p \right\rangle \\
&=& \displaystyle\sum_{r=0}^\infty \left\langle\what f(\pi_\lambda)e_q\:,\:e_r\right\rangle \left\langle\what f(\pi_\lambda)e_r\:,\:e_p\right\rangle \\
&=& \displaystyle\sum_{r=0}^\infty \what f(\pi_\lambda)_{r,q}\:\what f(\pi_\lambda)_{p,r}\:.
\end{eqnarray*}
Thus if $q \in \N$, by (\ref{pf1eq6}) and above
\begin{equation*}
\left(\what f(\pi_\lambda)^2\right)_{p,q} = 0 = \what f(\pi_\lambda)_{p,q}\:\what f(\pi_\lambda)_{0,0}\:.
\end{equation*}
For $q=0$, again by (\ref{pf1eq6}) and above
\begin{equation*}
\left(\what f(\pi_\lambda)^2\right)_{p,0} = \what f(\pi_\lambda)_{0,0}\:\what f(\pi_\lambda)_{p,0}\:,
\end{equation*}
and hence (\ref{pf1eq7}) is established for $n=2$.

Now assuming (\ref{pf1eq7}) for $n=j$, we see that 
\begin{eqnarray*}
\left(\what f(\pi_\lambda)^{j+1}\right)_{p,q} = \left\langle \what f(\pi_\lambda)\left(\what f(\pi_\lambda)^j e_q\right)\:,\:e_p \right\rangle &=& \left\langle \what f(\pi_\lambda)\left(\displaystyle\sum_{r=0}^\infty \left\langle\what f(\pi_\lambda)^j e_q\:,\:e_r\right\rangle e_r \right)\:,\:e_p \right\rangle \\
&=& \displaystyle\sum_{r=0}^\infty \left\langle\what f(\pi_\lambda)^j e_q\:,\:e_r\right\rangle \left\langle\what f(\pi_\lambda)e_r\:,\:e_p\right\rangle \\
&=& \displaystyle\sum_{r=0}^\infty \left(\what f(\pi_\lambda)^j\right)_{r,q}\:\what f(\pi_\lambda)_{p,r}\\
&=& \displaystyle\sum_{r=0}^\infty \what f(\pi_\lambda)_{r,q}\:\what f(\pi_\lambda)^{j-1}_{0,0}\:\what f(\pi_\lambda)_{p,r}\:.
\end{eqnarray*}
Now if $q \in \N$, by (\ref{pf1eq6}),
\begin{equation*}
\left(\what f(\pi_\lambda)^{j+1}\right)_{p,q} = 0 = \what f(\pi_\lambda)_{p,q}\:\what f(\pi_\lambda)^j_{0,0}\:,
\end{equation*}
and for $q=0$,
\begin{equation*}
\left(\what f(\pi_\lambda)^{j+1}\right)_{p,0} = \what f(\pi_\lambda)^j_{0,0}\:\what f(\pi_\lambda)_{p,0}\:,
\end{equation*}
this proves the claim (\ref{pf1eq7}). Similar result is true for $\Psi$ as well.

Now noting (\ref{pf1eq6}) and plugging (\ref{pf1eq7}) in (\ref{pf1eq5}), we obtain for $p \in \N \cup \{0\}$,
\begin{equation}\label{pf1eq8}
\what f(\pi_\lambda)_{p,0}\left(1 - \displaystyle\sum_{n=2}^N a_n \what f(\pi_\lambda)^{n-1}_{0,0}\right)= \what \Psi(\pi_\lambda)_{p,0}\left(1 - \displaystyle\sum_{n=2}^N a_n \what \Psi(\pi_\lambda)^{n-1}_{0,0}\right)\:.
\end{equation}
For $p=0$, (\ref{pf1eq8}) turns out to be 
\begin{equation}\label{pf1eq9}
\what f(\pi_\lambda)_{0,0} - \displaystyle\sum_{n=2}^N a_n \what f(\pi_\lambda)^n_{0,0}= \what \Psi(\pi_\lambda)_{0,0} - \displaystyle\sum_{n=2}^N a_n \what \Psi(\pi_\lambda)^n_{0,0}\:.
\end{equation}
We define
\begin{equation*}
\Omega:= \left\{\lambda \in \mathfrak{a} \mid \what f(\pi_\lambda)_{0,0}= \what \Psi(\pi_\lambda)_{0,0}\right\}\subset \mathfrak{a}\:,
\end{equation*}
and by adopting a topological argument, we show that $\Omega = \mathfrak{a}$. Since both $\lambda \mapsto \what f(\pi_\lambda)_{0,0}$ and $\lambda \mapsto \what \Psi(\pi_\lambda)_{0,0}$ are continuous, it follows that $\Omega$ is closed in $\mathfrak{a}$. To show that  $\Omega$ is also open in $\mathfrak{a}$, we first note that by non-negativity of $\Psi$ and the bound of the spherical functions (\ref{phi_lambda_bound}), we have for all $\lambda \in \mathfrak{a}$
\begin{equation} \label{pf1eq10}
\left|\what \Psi(\pi_\lambda)_{0,0}\right| = \left|\int_G \Psi(x)\:\varphi_{\lambda}(x)\:dx\right| \le  \int_G \Psi(x)\:\left|\varphi_{\lambda}(x)\right|\:dx \le  \int_G \Psi(x)\:dx\:.
\end{equation}
If $\Psi$ is a non-trivial element of $L^1(G/K)$, that is,
\begin{equation*}
E:=\{x \in G \mid \Psi(x)>0\}
\end{equation*}
is a set of positive Haar measure, then we claim that one actually has strict inequality above, that is,
\begin{equation} \label{pf1eq11}
\left|\what \Psi(\pi_\lambda)_{0,0}\right| < \int_G \Psi(x)\:dx\:,\:\text{ for all } \lambda \in \mathfrak{a}\:.
\end{equation}
Indeed, if there exists $\lambda_0 \in \mathfrak{a}$ such that
\begin{equation*}
\left|\what \Psi(\pi_{\lambda_0})_{0,0}\right| = \int_G \Psi(x)\:dx\:,
\end{equation*}
then one has equality throughout (\ref{pf1eq10}) for $\lambda_0$ and hence, in particular,
\begin{equation*}
\int_G \Psi(x)\:\left|\varphi_{\lambda_0}(x)\right|\:dx = \int_G \Psi(x)\:dx\:.
\end{equation*}
Then using the non-negativity of $\Psi$ and the definition of $E$, we get that
\begin{equation*}
\int_E \Psi(x)\:\left(\left|\varphi_{\lambda_0}(x)\right|-1\right)\:dx = 0\:.
\end{equation*}
Now positivity of $\Psi$ on $E$ along with the boundedness (\ref{phi_lambda_bound}) of spherical functions yield the existence of a set of positive Haar measure, say $F$, such that
\begin{equation*}
\left|\varphi_{\lambda_0}(x)\right|=1\:,\:\text{ for all } x \in F\:.
\end{equation*} 
Then the inequalities (\ref{phi_lambda_bound}) and (\ref{phi_lambda_decay}) yield for all $x \in F$,
\begin{equation*}
1 = \left|\varphi_{\lambda_0}(x)\right| \le \varphi_0(x) \le 1\:,
\end{equation*}
and hence 
\begin{equation*}
\varphi_0 \equiv 1\:, \text{ on } F\:.
\end{equation*}
Now as $F$ is a set of positive Haar measure and $\varphi_0$ being an eigenfunction of $\Delta$, is real-analytic, we must have that 
\begin{equation*}
\varphi_0 \equiv 1\:, \text{ on } G\:,
\end{equation*}
but this contradicts the decay of $\varphi_0$ away from the identity (\ref{phi_lambda_decay}). Hence we have the strict inequality (\ref{pf1eq11}), which combined with (\ref{pf1eq2}) yields that
\begin{equation*} 
\left|\what \Psi(\pi_\lambda)_{0,0}\right| < \int_G \Psi(x)\:dx \le t_{\Qc}\:,\:\text{ for all } \lambda \in \mathfrak{a}\:.
\end{equation*}
On the other hand, if $\Psi$ is zero almost everywhere, then trivially we have that
\begin{equation*} 
\left|\what \Psi(\pi_\lambda)_{0,0}\right| =0 < t_{\Qc}\:,\:\text{ for all } \lambda \in \mathfrak{a}\:.
\end{equation*}
Hence in either cases, we have
\begin{equation} \label{pf1eq12}
\left|\what \Psi(\pi_\lambda)_{0,0}\right| <  t_{\Qc}\:,\:\text{ for all } \lambda \in \mathfrak{a}\:.
\end{equation}

Now given $\lambda_0 \in \Omega$, there exists $\varepsilon>0$, such that
\begin{equation} \label{pf1eq13}
\left|\what \Psi(\pi_{\lambda_0})_{0,0}\right| < (1-\varepsilon)t_{\Qc}\:.
\end{equation}
By continuity of $\lambda \mapsto \what f(\pi_\lambda)_{0,0}$, there exists $\delta>0$ such that
\begin{equation}\label{pf1eq14}
\left|\left| \what f(\pi_{\lambda_0})_{0,0}\right|-\left| \what f(\pi_\lambda)_{0,0}\right|\right| < \varepsilon\: t_{\Qc}\:,\:\text{ for all } \lambda \in B(\lambda_0,\delta)\:,
\end{equation}
where $B(\lambda_0,\delta)$ is the ball in $\mathfrak{a}$ with center $\lambda_0$ and radius $\delta$ with respect to the norm on $\mathfrak{a}$ inherited from $\mathfrak{p}$. Now as $\lambda_0 \in \Omega$, by (\ref{pf1eq13}) we have
\begin{equation*}
\left|\what f(\pi_{\lambda_0})_{0,0}\right|=\left|\what \Psi(\pi_{\lambda_0})_{0,0}\right| < (1-\varepsilon)t_{\Qc}\:,
\end{equation*}
which combined with (\ref{pf1eq14}) yields
\begin{equation*}
\left|\what f(\pi_\lambda)_{0,0}\right| < t_{\Qc}\:,\:\text{ for all } \lambda \in B(\lambda_0,\delta)\:.
\end{equation*}
So in particular both
\begin{equation*}
\left|\what \Psi(\pi_\lambda)_{0,0}\right|<t_{\Qc}\:,\text{ and }\:\left|\what f(\pi_\lambda)_{0,0}\right| <t_{\Qc}\:,\:\text{ for all } \lambda \in B(\lambda_0,\delta)\:.
\end{equation*}
Then in view of (\ref{pf1eq9}) and the fact that the complex polynomial $\Qc$ is injective on the disk $\mathcal{D}$ in $\C$ with center $0$ and radius $t_{\Qc}$ (whose proof is contained in the proof of Theorem \ref{Rn-thm}), it follows that
\begin{equation*}
\what \Psi(\pi_\lambda)_{0,0} =\what f(\pi_\lambda)_{0,0}\:,\:\text{ for all } \lambda \in B(\lambda_0,\delta)\:.
\end{equation*}  
Thus $B(\lambda_0,\delta) \subset \Omega$ and hence $\Omega$ is also open in $\mathfrak{a}$. 

We finally show that $\Omega$ is non-empty. To show this, we consider $f^\#$, the left $K$-average of $f$,
\begin{equation*}
f^\#(g):=\int_K f(kg)\:dk\:,\text{ for } g \in G\:.
\end{equation*}
Now as $f \in L^1(G/K)$, we note that $f^\# \in L^1(G//K)$ and moreover, for all $\lambda \in \mathfrak{a}$,
\begin{equation*}
\what f(\pi_{\lambda})_{0,0} = \Hc f^\#(\lambda)\:,
\end{equation*}
the Spherical Fourier transform of $f^\#$. Then by the Riemann-Lebesgue lemma (\ref{Riemann-Lebesgue}) on $\Hc f^\#$, there exists $\lambda_0 \in \mathfrak{a}$ such that
\begin{equation*} 
\left|\what f(\pi_{\lambda_0})_{0,0}\right| <  t_{\Qc}\:.
\end{equation*} 
Thus by (\ref{pf1eq12}), we get that both $\what f(\pi_{\lambda_0})_{0,0}\:,\: \what \Psi(\pi_{\lambda_0})_{0,0} \in \mathcal{D}$ and then by (\ref{pf1eq9}) and the injectivity of $\Qc$ on $\mathcal{D}$, we obtain  $\what f(\pi_{\lambda_0})_{0,0}\:=\: \what \Psi(\pi_{\lambda_0})_{0,0}$, that is, $\lambda_0 \in \Omega$ and hence $\Omega$ is non-empty. Since $\Omega$ is a non-empty subset of $\mathfrak{a}$ which is both closed as well as open, we have $\Omega=\mathfrak{a}$, that is,
\begin{equation}\label{pf1eq16}
\what f(\pi_\lambda)_{0,0}\:=\: \what \Psi(\pi_\lambda)_{0,0}\:,\:\text{ for all } \lambda \in \mathfrak{a}\:.
\end{equation}
Now in view of (\ref{pf1eq12}), we reconsider the complex polynomial 
\begin{equation*}
\Qc(z) = z - \displaystyle\sum_{n=2}^N a_n z^n = z - \Pc(z)\:,
\end{equation*}
and now aim to show that the polynomial 
\begin{equation*}
\Qc_1(z):=\frac{\Qc(z)}{z}=1-\frac{\Pc(z)}{z}
\end{equation*}
has no zeroes in $\mathcal{D}$. To show this it suffices to prove that 
\begin{equation}\label{pf1eq17}
\left|\frac{\Pc(z)}{z}\right|<1\:,\:\text{ for all } z \in \mathcal{D}\:.
\end{equation}
We first note that $\Pc(0)=0$ and then recall (\ref{Rn-inj_eq2}) which states that
\begin{equation*} 
\left|\Pc'(z)\right| < 1\:,\text{ for all } z \in \mathcal{D}.
\end{equation*}
Then for any $z \in \mathcal{D}$, we get
\begin{equation*}
\left|\Pc(z)\right| = \left|\Pc(z)-\Pc(0)\right| =\left|\int_0^z \Pc'(\xi)\:d\xi\right|<|z|\:.
\end{equation*}
Thus we get (\ref{pf1eq17}) as
\begin{equation*}
\left|\frac{\Pc(z)}{z}\right|=\frac{\left|\Pc(z)\right|}{\left|z\right|}<1\:.
\end{equation*}
Hence $\Qc_1$ has no zeroes in $\mathcal{D}$. Now by plugging (\ref{pf1eq16}) in (\ref{pf1eq8}), we note that
\begin{equation*}
\what f(\pi_\lambda)_{p,0}\: \Qc_1\left(\what \Psi(\pi_\lambda)_{0,0}\right)=\what \Psi(\pi_\lambda)_{p,0}\: \Qc_1\left(\what \Psi(\pi_\lambda)_{0,0}\right)\:.
\end{equation*}
Then in view of (\ref{pf1eq12}) and the fact that $\Qc_1$ has no zeroes in $\mathcal{D}$, we get that
\begin{equation*}
\what f(\pi_\lambda)_{p,0}\:=\: \what \Psi(\pi_\lambda)_{p,0}\:,\:\text{ for all } \lambda \in \mathfrak{a}\:, p \in \N \cup \{0\}\:.
\end{equation*}
The above along with (\ref{pf1eq6}) yields (\ref{pf1eq3}) which completes the proof of Theorem \ref{thm1}.
\end{proof}

We now present the proof of Theorem \ref{thm2}:
\begin{proof}[Proof of Theorem \ref{thm2}]
For $\xi >0$, the resolvent operator $\left(\xi I - \La\right)^{-1}$ is realized by convolution on the right with the $K$-biinvariant tempered distribution $K_\xi$ on $G$ given by 
\begin{equation*} 
K_\xi(x)= \int_0^\infty e^{-t\xi}\:h_t(x)\:dt\:,
\end{equation*}
where $h_t$ is the heat kernel on $\X$. Hence $K_\xi$ is non-negative and moreover, an application of Fubini's theorem yields
\begin{equation} \label{pf2eq1}
\int_G K_\xi(x) dx = \int_G  \int_0^\infty e^{-t\xi}\:h_t(x)\:dt\:dx = \int_0^\infty e^{-t\xi}\:dt= \frac{1}{\xi}\:.
\end{equation}
Now the integro-differential equation (\ref{PDE}) can be rewritten as,
\begin{equation*} 
 u = {\left(-\La +\xi \right)}^{-m}\left\{V(1-u)\right\}+{\left(-\La +\xi \right)}^{-m}\left\{\mu \left(*^{m+1} u\right)\right\}\:.
\end{equation*}
Then as $K_\xi,V$ and $u$ all are $K$-biinvariant, the convolutions commute and we can write,
\begin{equation*}
u=\left(*^m K_\xi\right)*\left\{V(1-u)\right\} + \mu \left\{*^m \left(K_\xi * u\right)\right\}*u\:.
\end{equation*}
Now defining,
\begin{equation*}
f:= \mu^{1/m} \left(K_\xi *u\right)\:,
\end{equation*}
we note that $f \in L^1(G//K)$ is real-valued and then rewrite the above as,
\begin{equation} \label{pf2eq2}
u=\left(*^m K_\xi\right)*\left\{V(1-u)\right\} +  \left(*^m f\right)*u\:.
\end{equation}
Next by integrating out $f$ and plugging in the constraint (\ref{constraint}) and (\ref{pf2eq1}), we obtain
\begin{equation} \label{pf2eq3}
\int_G f = \mu^{1/m}\left(\int_G K_\xi\right)\left(\int_G u\right)=\frac{\mu^{1/m}}{\xi \delta}\:,
\end{equation}
and thus in view of the hypothesis (\ref{thm2_eq}), it follows that
\begin{equation*}
\int_G f \in (0,1)\:.
\end{equation*}
Now convolving (\ref{pf2eq2}) with $\mu^{1/m} K_\xi$, we get
\begin{equation*}
f= \mu^{1/m}\left(*^{m+1} K_\xi\right)*\left\{V(1-u)\right\} +  \left(*^{m+1} f\right)\:.
\end{equation*}
Then as $K_\xi$ is non-negative and from the hypothesis, we also have that $V$ is non-negative and $u \le 1$ almost everywhere on $G$, it follows that
\begin{equation*}
f \ge *^{m+1}f\:.
\end{equation*}
Then in view of (\ref{pf2eq3}), we note that Theorem \ref{thm1} is applicable for $f$ to conclude that $f$ is non-negative. Now define for any real-valued $h$,
\begin{equation*}
h_{-}(x):=\begin{cases}
	 -h(x)\:  &\text{ if } h(x)< 0  \\
	0\:  &\text{ if } h(x) \ge 0 \:,
	\end{cases}
\end{equation*}
and also, 
\begin{equation*}
h_{+}(x):=\begin{cases}
	 0\:  &\text{ if } h(x)< 0  \\
	h(x)\:  &\text{ if } h(x) \ge 0 \:,
	\end{cases}
\end{equation*}
and thus $h=h_+-h_-$. Then by (\ref{pf2eq2}) and non-negativity of $K_\xi\:,\: V\:,\: 1-u$ and $f$, it follows that
\begin{eqnarray*}
u_{-} &=& \left[\left(*^m K_\xi\right)*\left\{V(1-u)\right\} +  \left(*^m f\right)*u\right]_{-} \\
& \le &  \left[ \left(*^m f\right)*u\right]_{-} \\
&=& \left[ \left(*^m f\right)*u_+ - \left(*^m f\right)*u_-\right]_{-} \\
& \le & \left(*^m f\right)*u_-\:.
\end{eqnarray*}
Integrating the above inequality and then invoking (\ref{pf2eq3}) yields
\begin{equation*}
\int_G u_- \le {\left(\int_G f \right)}^m \left(\int_G u_- \right)  = \frac{\mu}{\xi^m \delta^m} \left(\int_G u_- \right) \:.
\end{equation*} 
The above inequality combined with the hypothesis (\ref{thm2_eq}) yields
\begin{equation*}
\int_G u_- = 0\:,
\end{equation*}
which in turn implies that $u_- = 0$ almost everywhere on $G$ and hence $u \ge 0$ almost everywhere on $G$. This completes the proof of Theorem \ref{thm2}.
\end{proof}

\section*{Acknowledgements} 
The author is thankful to Prof. Swagato K. Ray for bringing \cite{CJLL} to his attention. The author is supported by a research fellowship of Indian Statistical Institute.

\bibliographystyle{amsplain}

\end{document}